\documentclass[11pt,fullpage, doublespace]{amsart}
\usepackage{amsfonts, amssymb, amsmath}
\usepackage[all]{xy}
\usepackage{bm}
\usepackage{authblk}
\usepackage{graphicx}

\newtheorem{theorem}{Theorem}[section]

\newtheorem{lemma}{Lemma}[section]

\newtheorem{corr}{Corollary}[section]

\newcommand{\ld}{\lambda}

\newcommand{\Z}{\mathbb{Z}}

\title{Explicit constants in averages involving the
multiplicative order}
\author{Kim, Sungjin}
\affil[1]{Department of Mathematics \\ University of California, Los Angeles \\ Math Science Building 6617A \\ E-mail:  707107@gmail.com }

\makeatletter
\let\authors\AB@authors
\makeatother

\begin{document}
    \maketitle

    \begin{abstract}
    Let $a>1$. Denote by $\ell_a(p)$ the multiplicative order of $a$ modulo $p$. We look for an estimate of sum of $\frac{\ell_a(p)}{p-1}$ over primes $p\leq x$ on average. When we average over $a\leq N$, we observe a statistic of $C\mathrm{Li}(x)$. P. J. Stephens ~\cite[Theorem 1]{S} proved this statistic for $N>\exp(c_1\sqrt{\log x})$ for some positive constant $c_1$. Upon this result, the value of $c_1$ is at least $12e^9$ in ~\cite{S}, we reduce this value of $c_1$ to $3.42$ by a different method. In fact, ~\cite[Theorem 1, 3]{S} hold with $N>\exp(3.42\sqrt{\log x})$, and ~\cite[Theorem 2, 4]{S} hold with $N>\exp(4.8365\sqrt{\log x})$. Also, we improve the range of $y$, from $y\geq \exp((2+\epsilon)\sqrt{\log x\log\log x})$ in ~\cite[Theorem 1]{LP}, to $y > \exp(3.42\sqrt{\log x})$.
        \end{abstract}

\section{Introduction}\footnote{Keywords: Artin, Primitive Root, Average, AMS Subject Classification Code: 11A07}
    We use $p$, $q$ to denote prime numbers, and use $c_i$ to denote an absolute positive constants.
    Let $a\geq 1$ be an integer. Denote by $\ell_a(p)$ the multiplicative order of $a$ modulo prime $p$. Artin's Conjecture on Primitive Roots (AC) states that for non-square non-unit $a$, $a$ is a primitive root modulo $p$ for infinitely many primes $p$. Thus, $\ell_a(p)=p-1$ for infinitely many primes $p$. Assuming the Generalized Riemann Hypothesis (GRH) for Dedekind zeta functions for Kummer extensions, C. Hooley ~\cite{H} showed that $\ell_a(p)=p-1$ for positive proportion of primes $p\leq x$. Thus, we expect that $\ell_a(p)$ is large, and close to $p-1$ for large number of primes $p$. On average, we expect that $\ell_a(p)/(p-1)$ behaves like a constant. P. J. Stephens (see ~\cite[Theorem 1]{S}) showed that if $N>\exp(c_1 \sqrt{\log x})$ then for any positive constant $A$,
    $$
    N^{-1}\sum_{a\leq N}\sum_{p\leq x}\frac{\ell_a(p)}{p-1}=C\textrm{Li}(x)+O\left(\frac{x}{\log^A x}\right),
    $$
    where $C$ is Stephens' constant:
    $$
    C=\prod_p\left(1-\frac p{p^3-1}\right).$$
    Although the value of the positive constant $c_1$ is not explicitly given in ~\cite{S}, but we see that $c_1$ is at least $12e^9$ (see ~\cite[Lemma 7]{S}). The optimal value of $c_1$ along Stephens' method is any positive number greater than $2\sqrt2 e\approx 7.6885$. We replace $7.6885$ by $3.42$ by a different method.
    \begin{theorem}
    If $N>\exp(3.42\sqrt{\log x})$, then for any positive constant $D$,
    \begin{equation}
    N^{-1}\sum_{a\leq N}\sum_{p\leq x}\frac{\ell_a(p)}{p-1}=C \mathrm{Li}(x)+O\left(\frac{x}{\log^D x}\right).
    \end{equation}
    \end{theorem}
    Similarly, we give an explicit constant $c_1$ in ~\cite[Theorem 2]{S}.
    \begin{theorem}
    If $N>\exp(4.8365 \sqrt{\log x})$, then for any positive constant $E$,
    \begin{equation}
    N^{-1}\sum_{a\leq N}\left(\sum_{p<x}\frac{\ell_a(p)}{p-1}-C\mathrm{Li}(x)\right)^2\ll \frac{x^2}{\log^E x}.
    \end{equation}
    \end{theorem}
    Subsequently, we obtain an improvement of average results on $$P_a(x) = \{p\leq x | a \textrm{ is a primitive root modulo $p$.}\}$$ (see ~\cite[Theorem 1, 2]{S0}):

    If $N>\exp(3.42 \sqrt{\log x})$, then for any positive constant $D$,
    \begin{equation}
    N^{-1}\sum_{a\leq N}P_a(x) = A\pi(x) + O\left(\frac{x}{\log^D x}\right)
    \end{equation}
    where $A=\prod_p \left(1-\frac{1}{p(p-1)}\right)$ is Artin's constant. Also, the normal order result:

    If $N>\exp(4.8365 \sqrt{\log x})$, then for any positive constant $E$,
    \begin{equation}
    N^{-1}\sum_{a\leq N}\left(P_a(x) - A\pi(x) \right)^2 \ll \frac{x^2}{\log^E x}.
    \end{equation}

    Stephens also proved that the average number of prime divisors of $a^n-b$ also asymptotic to $C\textrm{Li}(x)$ in ~\cite[Theorem 3]{S}, and proved normal order result in ~\cite[Theorem 4]{S}. The number $N$ is rather large compared to those of Theorem 1, 2. ($N>x(\log x)^{c_2}$ in ~\cite[Theorem 3]{S}, and $N>x^2(\log x)^{c_2}$ in ~\cite[Theorem 4]{S} respectively.) He mentioned that these could probably be improved by using the large sieve inequality as in Theorem 1, 2. However, he did not carry out the improvement in ~\cite{S}. Here, we state the improvement and prove them.
    \begin{theorem}
    If $N>\exp(3.42\sqrt{\log x})$, then for any positive constant $D$,
    \begin{equation}
    N^{-2}\sum_{a\leq N}\sum_{b\leq N}\sum_{\substack{{p\leq x}\\{p|a^n-b}\\{\textrm{for some $n$}}}}1 =C\mathrm{Li}(x)+O\left(\frac{x}{\log^D x}\right).
    \end{equation}
    \end{theorem}
    \begin{theorem}
    If $N>\exp(4.8365 \sqrt{\log x})$, then for any positive constant $E$,
    \begin{equation}
    N^{-2}\sum_{a\leq N}\sum_{b\leq N}\left(\sum_{\substack{{p\leq x}\\{p|a^n-b}\\{\textrm{for some $n$}}}}1-C\mathrm{Li}(x)\right)^2\ll \frac{x^2}{\log^E x}.
    \end{equation}
    \end{theorem}
    In ~\cite{C}, Carmichael's lambda function $\ld(n)$ is defined by the exponent of the group $(\Z/n\Z)^{*}$. We say that $a$ is a $\ld$-primitive root modulo $n$ if the order $\ell_a(n)$ of $a$ modulo $n$ is exactly $\ld(n)$. Following the definitions and notations in ~\cite{LP},
    $$R(n)=\# \{a\in (\Z/n\Z)^{*}  :  a \textrm{ is a $\ld$-primitive root modulo $n$.} \},$$
    $$N_a(x)=\# \{ n\leq x : a \textrm{ is a $\ld$-primitive root modulo $n$.}\}. $$
    S. Li ~\cite{L} proved that for $y\geq \exp((\log x)^{3/4})$,
    \begin{equation}
    \frac1y \sum_{a\leq y} N_a(x) \sim \sum_{n\leq x} \frac{R(n)}n.
    \end{equation}
    This was further improved by S. Li and C. Pomerance ~\cite{LP}, where they obtained (7) with range of $y$:
    \begin{equation}
    y\geq \exp((2+\epsilon)\sqrt{\log x \log\log x}).
    \end{equation}
    We prove (7) with a wider range of $y$:
    \begin{theorem}
    If $y>\exp(3.42\sqrt{\log x})$, then there exists a positive constant $c_2$ such that
    \begin{equation}
    \frac1y \sum_{a\leq y}N_a(x) = \sum_{n\leq x} \frac{R(n)}n + O\left(x\exp(-c_2\sqrt{\log x})\right).
    \end{equation}
    \end{theorem}
    \section{Proof of Theorems}
    \subsection{Proof of Theorem 1.1 and 1.2}
    We begin with the following elementary lemma (see ~\cite{B}):
    \begin{lemma}
    Let $r\geq 1$ and define $\tau_r(a)$ to be the number of ways to write $a$ as an ordered product of $r$ positive integers. If $N\geq 1$, then we have
    \begin{equation}
    \sum_{a\leq N}\tau_r(a) \leq \frac{1}{(r-1)!} N (\log N + r-1)^{r-1}.\end{equation}
    \end{lemma}
    The proof is by induction.
    \begin{corr}
    Let $c>0$. If $N\geq 1$ and $r-1\leq c\log N$, then
    \begin{equation}
    \sum_{a\leq N}\tau_r(a) \leq \frac{(1+c)^{r-1}}{(r-1)!}N \log^{r-1} N.
    \end{equation}
    \end{corr}
    We define $\tau_r'(a)$ to be the number of ways of writing $a$ as ordered product of $r$ positive integers, each of which does not exceed $N$.
    \begin{lemma}
    We have
    \begin{equation}
    \sum_{a\leq N^r} (\tau_r'(a))^2 \leq \left(\sum_{a\leq N}\tau_r(a)\right)^r.
    \end{equation}
    \end{lemma}
    \begin{proof}
    \begin{align*}
    \sum_{a\leq N^r} (\tau_r'(a))^2 &= \sum_{a_1, \cdots, a_r \leq N} \tau_r'(a_1 \cdots a_r)
    \leq  \sum_{a_1, \cdots, a_r \leq N} \tau_r'(a_1) \cdots \tau_r'(a_r)\\
    &=\left(\sum_{a\leq N} \tau_r'(a)\right)^r
    =\left(\sum_{a\leq N} \tau_r (a)\right)^r.
    \end{align*}
    \end{proof}
    \begin{corr}
    Let $c>0$. If $N\geq 1$ and $r-1\leq c\log N$, then
    \begin{equation}
    \sum_{a\leq N^r} (\tau_r'(a))^2 \leq \left(\frac{(1+c)^{r-1}}{(r-1)!}N \log^{r-1} N\right)^r.
    \end{equation}
    \end{corr}
    We follow the notations in ~\cite{S}, and give a critical upper estimate of the following quantities:
    \begin{lemma}
    Let
    \begin{equation}
    S_4 = \sum_{p\leq x}\sum_{\chi (\mathrm{  mod \  }p)}^{ \ \ \ \ *} \frac{1}{\mathrm{ord} (\chi)}\left|\sum_{a\leq N}\chi(a)\right|
    \end{equation}
    and
    \begin{equation}
    S_{10}= \sum_{p\leq x}\sum_{\substack{{q\leq x}\\{p\neq q}}} \sum_{\chi (\mathrm{ mod \ }pq)}^{ \ \ \ \ *} \frac{1}{\mathrm{ord} (\chi)}\left|\sum_{a\leq N}\chi(a)\right|.
    \end{equation}
    The sum $\sum^{*}$ denotes the sum over non-principal primitive characters.
    Then by H\"{o}lder inequality and the large sieve inequality,
    \begin{equation}
    S_4 \ll x^{1-\frac{1}{2r}} (x^{\frac{1}{r}} + N^{\frac12})\left(\sum_{a\leq N^r}(\tau_r'(a))^2\right)^{\frac1{2r}}
    \end{equation}
    and
    \begin{equation}
    S_{10} \ll (x^2)^{1-\frac{1}{2r}} (x^{\frac{2}{r}} + N^{\frac12})\left(\sum_{a\leq N^r}(\tau_r'(a))^2\right)^{\frac1{2r}}.
    \end{equation}
    \end{lemma}
    \begin{corr}
    By taking $N^{r-1}< x^2\leq N^r$, we have for $r-1 \leq c\log N$,
    \begin{equation}
    S_4 \ll xN^{\frac14+\frac1{4r}} \left(\frac{(1+c)^{r-1}}{(r-1)!}N \log^{r-1} N\right)^{\frac12} = xN^{\frac34+\frac1{4r}} \left(\frac{(1+c)^{r-1}}{(r-1)!} \log^{r-1} N\right)^{\frac12}.
    \end{equation}
    By taking $N^{r-1}<x^4\leq N^r$, we have for $r-1 \leq c\log N$,
    \begin{equation}
    S_{10}\ll x^2N^{\frac14+\frac1{4r}} \left(\frac{(1+c)^{r-1}}{(r-1)!}N \log^{r-1} N\right)^{\frac12}=x^2N^{\frac34+\frac1{4r}} \left(\frac{(1+c)^{r-1}}{(r-1)!} \log^{r-1} N\right)^{\frac12}.
    \end{equation}
    \end{corr}
    By ~\cite{S0}, we may assume that $N=\exp(K\sqrt{\log x})$.  Then $N^{\frac1{4r}}=\exp(O(K^2))= (\log x)^{O(1)}$ and we have the following:
    \begin{equation}
    \log N=K \sqrt{\log x},
    \end{equation}
    \begin{equation}
    \log\log N = \log K + \frac12 \log\log x,
    \end{equation}
    \begin{equation}
    r-1<\frac{2\log x}{\log N} = \frac2K  \sqrt{\log x}\leq r.
    \end{equation}
    By Stirling's formula, we have
    \begin{align*}
    &S_4 \ll x(\log x)^{O(1)} N \exp\left( -\frac14 \log N  + \frac12 (r-1) \log (1+c) - \frac12 \log (r-1)!+\frac{r-1}2\log\log N \right)\\
    &\ll x N  \exp\left( \sqrt{\log x} \left( -\frac K4  + \frac1K   \log(1+c) - \frac1K \log 2 +\frac1K  + \frac{2\log K}K +o(1)\right)\right).
    \end{align*}
    From $r-1\leq c\log N$, we have $\frac2K \leq cK$. Taking $c=\frac2{K^2}$, we have
    \begin{equation}
    S_4 \ll x N  \exp\left(\sqrt{\log x} \left( -\frac K4  + \frac1K   \log\left(1+\frac2{K^2}\right) - \frac1K \log 2 +\frac1K  + \frac{2\log K}K +o(1)\right)\right).
    \end{equation}
    If $K\geq 3.42$, then we see that
    $$
    f_1(K)=-\frac K4  + \frac1K   \log\left(1+\frac2{K^2}\right) - \frac1K \log 2 +\frac1K  + \frac{2\log K}K  <0.$$
    We finally obtain that
    \begin{lemma}
    If $N>\exp(3.42\sqrt{\log x})$, then there is a positive constant $c_2$ such that
    \begin{equation}
    S_4 \ll x N \exp\left(-c_2 \sqrt{\log x}\right).
    \end{equation}
    \end{lemma}
    Now, we deal with $S_{10}$. Let $N=\exp(K\sqrt{\log x})$. By ~\cite{S0}, we may assume that $K\leq 6 \sqrt{\log\log x}$.  Then $N^{\frac1{4r}}=\exp(O(K^2))= (\log x)^{O(1)}$ and we have the following:
    \begin{equation}
    \log N=K \sqrt{\log x},
    \end{equation}
    \begin{equation}
    \log\log N = \log K + \frac12 \log\log x,
    \end{equation}
    \begin{equation}
    r-1<\frac{4\log x}{\log N} = \frac4K \sqrt{\log x}\leq r.
    \end{equation}
    By Stirling's formula, we have
    \begin{align*}
    &S_{10} \ll x^2 (\log x)^{O(1)}N \exp\left( -\frac14 \log N  + \frac12 (r-1) \log (1+c) - \frac12 \log (r-1)!+\frac{r-1}2\log\log N \right)\\
    &\ll x^2 N \exp\left( \sqrt{\log x} \left( -\frac K4  + \frac2K   \log(1+c) - \frac2K \log 4 +\frac2K  + \frac{4\log K}K+o(1) \right)\right).
    \end{align*}
    From $r-1\leq c\log N$, we have $\frac4K \leq cK$. Taking $c=\frac4{K^2}$, we have
    \begin{equation}
    S_{10} \ll x^2 N  \exp\left( \sqrt{\log x} \left( -\frac K4  + \frac2K   \log\left(1+\frac4{K^2}\right) - \frac2K \log 4 +\frac2K  + \frac{4\log K}K +o(1)\right)\right).
    \end{equation}
    If $K\geq 4.8365$, then we see that
    $$
    f_2(K)=-\frac K4  + \frac2K   \log\left(1+\frac4{K^2}\right) - \frac2K \log 4 +\frac2K  + \frac{4\log K}K  <0.$$
    Therefore, we obtain that
    \begin{lemma}
    If $N>\exp(4.8365\sqrt{\log x})$, then there is a positive constant $c_3$ such that
    \begin{equation}
    S_{10} \ll x^2 N \exp\left(-c_3 \sqrt{\log x}\right).
    \end{equation}
    Denote by $\tilde{S}_{10}$ the following character sum:
    \begin{equation}
    \tilde{S}_{10}= \sum_{p\leq x}\sum_{\substack{{q\leq x}\\{p\neq q}}} \sum_{\chi (\mathrm{ mod \ }pq)}^{ \ \ \ \ \star} \frac{1}{\mathrm{ord} (\chi)}\left|\sum_{a\leq N}\chi(a)\right|
    \end{equation}
    where $\Sigma^{\star}$ is over non-principal characters. Then
    \begin{equation}
    \tilde{S}_{10}\ll x^2 N \exp\left(-c_3 \sqrt{\log x}\right).
    \end{equation}
    \end{lemma}
    For the estimate for $\tilde{S}_{10}$, note that
    $$
    \tilde{S}_{10} \ll S_{10}+ x S_4.$$
    \begin{proof}[Proof of Theorem 1.1]
    As in ~\cite[Theorem 1]{S}, we define a character sum $c_r(\chi)$ where $\chi$ is a Dirichlet character modulo $p$:
    For $r|p-1$, define
    \begin{equation}
    c_r(\chi)=\frac1{p-1}\sum_{\substack{ a<p \\ \ell_a (p)=\frac{p-1}r}} \chi(a).
    \end{equation}
    Then we have
    \begin{align*}
    S_3 &= N^{-1}\sum_{a\leq N}\sum_{p\leq x}\sum_{\substack{{w|p-1}\\{\ell_a(p)=\frac{p-1}w}}} w^{-1}
    =N^{-1}\sum_{a\leq N}\sum_{p\leq x}\sum_{w|p-1} w^{-1}\sum_{\chi (\mathrm{ mod \ } p)}c_w(\chi)\chi(a)\\
    &=N^{-1}\sum_{p\leq x}\sum_{w|p-1}w^{-1}\sum_{\chi (\mathrm{ mod \ } p)}c_w(\chi)\sum_{a\leq N}\chi(a)\\
    &=N^{-1}\sum_{p\leq x}\sum_{w|p-1} \frac{ \phi(\frac{p-1}w)}{w(p-1)}\left(N+O(1)+O\left(\frac Np\right)\right)+O\left(N^{-1}\sum_{p\leq x}\sum_{w|p-1}w^{-1}\sum_{\chi (\mathrm{ mod \ } p)}^{\ \ \ \ *}|c_w(\chi)|\left|\sum_{a\leq N}\chi(a)\right|\right) \\
    &=\sum_{p\leq x}\sum_{w|p-1}\frac{ \phi(\frac{p-1}w)}{w(p-1)} +O\left( \frac{x}{N\log x}\right) + O(\log\log x)+ O\left(\frac{S_4 \log x\log\log x}N\right)\\
    &=C\textrm{Li}(x) + O\left(\frac{x}{\log^A x}\right) + O\left(x\exp\left(-c_4 \sqrt{\log x}\right)\right)
    \end{align*}
    by Lemma 2.4 and ~\cite[Lemma 12]{S}. This completes the proof, since the second error term is dominated by the first.
    \end{proof}
    \begin{proof}[Proof of Theorem 1.2]
    As in ~\cite[Theorem 2]{S}, let
    \begin{equation}
    S_5=N^{-1}\sum_{a\leq N}\left(\sum_{p\leq x}\frac{\ell_a(p)}{p-1}-C\textrm{Li}(x)\right)^2,
    \end{equation}
    then
    \begin{equation}
    S_5= N^{-1}\sum_{p\leq x}\sum_{\substack{{q\leq x}\\{p\neq q}}}\sum_{a\leq N}\frac{\ell_a(p)\ell_a(q)}{(p-1)(q-1)} - C^2 \textrm{Li}^2(x) +O\left(\frac{x^2}{\log^E x}\right).
    \end{equation}
    Let
    \begin{equation}
    S_6=N^{-1}\sum_{p\leq x}\sum_{\substack{{q\leq x}\\{p\neq q}}}\sum_{a\leq N}\frac{\ell_a(p)\ell_a(q)}{(p-1)(q-1)}.
    \end{equation}
    Stephens decomposed $S_6$ into three parts:
    \begin{align*}
    S_6&=N^{-1}\sum_{p\leq x}\sum_{\substack{{q\leq x}\\{p\neq q}}}\sum_{w|p-1}\sum_{t|q-1}\frac1{wt}\sum_{\chi_1 (\mathrm{ mod \  }p)}c_w(\chi_1)\sum_{\chi_2 (\mathrm{ mod \ }q)}c_t(\chi_2)\sum_{a\leq N}\chi_1\chi_2(a)\\
    &=S_7+2S_8+S_9\end{align*}
    where $S_7$, $S_8$, $S_9$ are the contributions to $S_6$ when both $\chi_1$, $\chi_2$ are principal, when one of $\chi_1$, $\chi_2$ is principal, and when neither $\chi_1$, $\chi_2$ is principal, respectively. Then
    \begin{equation}
    S_7 = C^2 \textrm{Li}^2 (x) + O\left(\frac{x^2}{\log^E x}\right) + O\left(\frac{x^2}{N\log^2 x}\right),
    \end{equation}
    \begin{equation}
    S_8 \ll \frac{x^2}{\log^E x},
    \end{equation}
    and
    \begin{equation}
    S_9 \ll \frac{S_{10} \log^2 x\log\log x}N.
    \end{equation}
    Then we have by Lemma 2.5,
    \begin{align*}
    S_5 &\ll C^2\textrm{Li}^2 (x)-C^2\textrm{Li}^2 (x) + O\left(\frac{x^2}{\log^E x}\right) + O\left(\frac{x^2}{N\log^2 x}\right)+ O\left(x^2\exp\left(-c_5 \sqrt{\log x}\right)\right).
    \end{align*}
    Since the last two error terms are dominated by $\frac{x^2}{\log^E x}$, this completes the proof of Theorem 1.2.
    \end{proof}
    \subsection{Proof of Theorem 1.3 and 1.4}
    \begin{proof}[Proof of Theorem 1.3]
    As in ~\cite[Theorem 3]{S}, we write the sum as
    \begin{align*}
    N^{-2}\sum_{b\leq N}\sum_{p\leq x}\sum_{\substack{{a\leq N}\\{\ell_b(p)|\ell_a(p)}}}1&=N^{-2}\sum_{b\leq N}\sum_{p\leq x}\sum_{w|p-1}\sum_{\substack{{a\leq N}\\{\ell_b(p)|\frac{p-1}w}\\{\ell_a(p)=\frac{p-1}w}}}1
    =N^{-2}\sum_{b\leq N}\sum_{p\leq x}\sum_{\substack{{w|p-1}\\{\ell_b(p)|\frac{p-1}w}}}\sum_{\substack{{a\leq N}\\{\ell_a(p)=\frac{p-1}w}}}1\\
    &=N^{-2}\sum_{b\leq N}\sum_{p\leq x}\sum_{\substack{{w|p-1}\\{\ell_b(p)|\frac{p-1}w}}}\sum_{\chi(\mathrm{ mod \  }p)}c_w(\chi)\sum_{a\leq N}\chi(a)\\
    &=N^{-2}\sum_{b\leq N}\sum_{p\leq x}\sum_{\substack{{w|p-1}\\{\ell_b(p)|\frac{p-1}w}}}\frac{\phi\left(\frac{p-1}w\right)}{p-1}\left(N+O(1)+\left(\frac Np\right)\right)\\ & \ \ +O\left(N^{-2}\sum_{b\leq N}\sum_{p\leq x}\sum_{ w|p-1 }\sum_{\chi(\mathrm{ mod \  }p)}^{ \ \ \ \ *}|c_w(\chi)|\left|\sum_{a\leq N}\chi(a)\right|\right)\\
    &=N^{-1}\sum_{b\leq N}\sum_{p\leq x}\sum_{\substack{{w|p-1}\\{\ell_b(p)|\frac{p-1}w}}}\frac{\phi\left(\frac{p-1}w\right)}{p-1} + O\left(\frac{x}{N\log x}\right)+O(\log\log x) \\
    & \ \ + O\left(N^{-1}\log\log x\sum_{p\leq x}\tau_2(p-1)\sum_{\chi(\mathrm{ mod \ }p)}^{ \ \ \ \ *}\frac1{\textrm{ord} \chi}\left|\sum_{a\leq N}\chi(a)\right|\right).
    \end{align*}
    To treat the last error term, let $c_2$ be the positive constant in Lemma 2.4. Choose any positive constant $c_6$ smaller than $c_2$, and split the sum into two parts:
    \begin{equation}
    N^{-1}\sum_{p\leq x}\tau_2(p-1)\sum_{\chi(\mathrm{ mod \ }p)}^{ \ \ \ \ *}\frac1{\textrm{ord} \chi}\left|\sum_{a\leq N}\chi(a)\right| = \Sigma_1 + \Sigma_2,
    \end{equation}
    where $\Sigma_1$ is the sum over $p$'s with $\tau_2(p-1)< \exp(c_6 \sqrt{\log x})$, and $\Sigma_2$ is the sum over remaining $p$'s.  Then by Lemma 2.4,
    \begin{align*}
    \Sigma_1 &\ll N^{-1} \exp(c_6 \sqrt{\log x}) xN \exp(-c_2\sqrt{\log x}) \\
    &\ll x\exp(-c_7 \sqrt{\log x}).
    \end{align*}
    By an elementary estimate  $\sum_{n\leq x} \tau_2(n)^3 \ll x \log^7 x$,
    \begin{align*}
    \Sigma_2 &\ll N^{-1} \sum_{\substack{{p\leq x}\\{\tau_2(p-1)\geq \exp(c_6\sqrt{\log x})}}}\tau_2(p-1)^2 N\\
    &\ll \sum_{p\leq x} \frac{\tau_2(p-1)^3}{\exp(c_6\sqrt{\log x})}\\
    &\ll x\exp(-c_7 \sqrt{\log x}).
    \end{align*}
    Thus, we obtain
    \begin{equation}
    N^{-2}\sum_{b\leq N}\sum_{p\leq x}\sum_{\substack{{a\leq N}\\{\ell_b(p)|\ell_a(p)}}}1 = N^{-1}\sum_{b\leq N}\sum_{p\leq x}\sum_{\substack{{w|p-1}\\{\ell_b(p)|\frac{p-1}w}}}\frac{\phi\left(\frac{p-1}w\right)}{p-1} + O\left(x\exp(-c_7 \sqrt{\log x})\right).
    \end{equation}
    Now, we treat the first sum on the right side. Again, by writing it with character sums,
    \begin{align*}
    N^{-1}\sum_{b\leq N}\sum_{p\leq x}\sum_{\substack{{w|p-1}\\{\ell_b(p)|\frac{p-1}w}}}\frac{\phi\left(\frac{p-1}w\right)}{p-1}&= N^{-1}\sum_{p\leq x}\sum_{w|p-1} \frac{\phi\left(\frac{p-1}w\right)}{p-1}\sum_{t|\frac{p-1}w}\sum_{\substack{{b\leq N}\\{\ell_b(p)=\frac{p-1}{tw}}}}1\\
    &=N^{-1}\sum_{p\leq x}\sum_{w|p-1}\frac{\phi\left(\frac{p-1}w\right)}{p-1}\sum_{t|\frac{p-1}w}\sum_{\chi (\mathrm{ mod \ }p)}c_{tw}(\chi)\sum_{b\leq N}\chi(b)\\
    &=N^{-1}\sum_{p\leq x}\sum_{w|p-1}\frac{\phi\left(\frac{p-1}w\right)}{p-1}\sum_{t|\frac{p-1}w} \frac{\phi\left(\frac{p-1}{tw}\right)}{p-1}\left(N+O(1)+O\left(\frac Np\right)\right)+E\\
    &=\sum_{p\leq x}\sum_{w|p-1}\frac{\phi\left(\frac{p-1}w\right)}{p-1}\sum_{t|\frac{p-1}w} \frac{\phi\left(\frac{p-1}{tw}\right)}{p-1}+O\left(\frac{x}{N\log x}\right)+O(\log\log x) + E,
    \end{align*}
    where
    \begin{equation}
    E=N^{-1}\sum_{p\leq x}\sum_{w|p-1}\frac{\phi\left(\frac{p-1}w\right)}{p-1}\sum_{t|\frac{p-1}w}\sum_{\chi (\mathrm{ mod \ }p)}^{ \ \ \ \ *}c_{tw}(\chi)\sum_{b\leq N} \chi(b).
    \end{equation}
    We split the sum as before,
    \begin{align*}
    E&\ll N^{-1} \log\log x\sum_{p\leq x}\sum_{w|p-1}\tau_2(w)\sum_{\chi (\mathrm{ mod \ }p)}^{ \ \ \ \ *}\frac1{\textrm{ord} \chi}\left|\sum_{b\leq N}\chi(b)\right|\\
    &= N^{-1}\log\log x\sum_{p\leq x}\tau_3(p-1)\sum_{\chi (\mathrm{ mod \ }p)}^{\ \ \ \ *}\frac1{\textrm{ord} \chi}\left|\sum_{b\leq N}\chi(b)\right|\\
    &=\Sigma_3 + \Sigma_4,
    \end{align*}
    where $\Sigma_3$ is over $p$'s with $\tau_3(p-1)<\exp(c_6 \sqrt{\log x})$, and $\Sigma_4$ is over remaining $p$'s. By the same argument, we have
    \begin{equation}
    \Sigma_3+\Sigma_4 \ll x \exp(-c_7 \sqrt{\log x}). \end{equation}
    Therefore,
    \begin{equation}
    N^{-2}\sum_{b\leq N}\sum_{p\leq x}\sum_{\substack{{a\leq N}\\{\ell_b(p)|\ell_a(p)}}}1 = \sum_{p\leq x}\sum_{w|p-1}\frac{\phi\left(\frac{p-1}w\right)}{p-1}\sum_{t|\frac{p-1}w} \frac{\phi\left(\frac{p-1}{tw}\right)}{p-1}+ O\left(x\exp(-c_7\sqrt{\log x})\right).
    \end{equation}
    By the elementary identity $\sum_{d|n}\phi(d) = n$, we have
    \begin{equation}
    N^{-2}\sum_{b\leq N}\sum_{p\leq x}\sum_{\substack{{a\leq N}\\{\ell_b(p)|\ell_a(p)}}}1 =\sum_{p\leq x}\sum_{w|p-1}\frac{\phi\left(\frac{p-1}w\right)}{w(p-1)} + O\left(x\exp(-c_7\sqrt{\log x})\right).
    \end{equation}
    Then Theorem 1.3 follows by ~\cite[Lemma 12]{S}.
    \end{proof}
    \begin{proof}[Proof of Theorem 1.4]
    Using the same argument as for Theorem 1.2 we deduce that
    \begin{align*}
    N^{-2}\sum_{a\leq N}&\sum_{b\leq N}\left(\sum_{\substack{{p\leq x}\\{p|a^n-b}\\{\textrm{for some $n$}}}}1-C\textrm{Li}(x)\right)^2 = N^{-2}\sum_{a\leq N}\sum_{b\leq N}\sum_{\substack{{p\leq x}\\{p|a^n-b}\\{\textrm{for some $n$}}}}\sum_{\substack{{q\leq x}\\{q\neq p}\\{q|a^m-b}\\{\textrm{for some $m$}}}}1 -C^2 \textrm{Li}^2(x)+O\left(\frac{x^2}{\log^E x}\right)\\
    &=N^{-2}\sum_{a\leq N}\sum_{b\leq N}\sum_{\substack{{p\leq x}\\{\ell_b(p)|\ell_a(p)}}}\sum_{\substack{{q\leq x}\\{q\neq p}\\{\ell_b(q)|\ell_a(q)}}}1 -C^2 \textrm{Li}^2(x)+O\left(\frac{x^2}{\log^E x}\right)\\
    &=N^{-2}\sum_{p\leq x}\sum_{\substack{{q\leq x}\\{q\neq p}}}\sum_{\substack{{w|p-1}\\{t|\frac{p-1}w}}}\sum_{\substack{{u|q-1}\\{s|\frac{q-1}u}}}\sum_{\substack{{\chi_1, \chi_2(\mathrm{ mod \ }p)}\\{\chi_3, \chi_4(\mathrm{ mod \ }q)}}}c_w(\chi_1)c_{wt}(\chi_2)c_u(\chi_3)c_{us}(\chi_4)\sum_{a\leq N}\chi_1\chi_3(a)\sum_{b\leq N}\chi_2\chi_4(b)\\ & \ \ - C^2\textrm{Li}^2(x)+O\left(\frac{x^2}{\log^E x}\right).
    \end{align*}
    We prove that if any one of $\chi_i$, $i=1,2,3,4$ is non-principal, then they contribute $O(x^2/\log^E x)$. In this case, either $\chi_1\chi_3$ or $\chi_2\chi_4$ is non-principal. Suppose that $\chi_1\chi_3$ is non-principal. Then, the contribution is
    \begin{align*}
    &\ll N^{-1}(\log\log x)^4\sum_{p\leq x}\sum_{\substack{{q\leq x}\\{q\neq p}}}\sum_{\substack{{w|p-1}\\{t|\frac{p-1}w}}}\sum_{\substack{{u|q-1}\\{s|\frac{q-1}u}}}\sum_{\substack{{\chi_1, \chi_2(\mathrm{ mod \ }p)}\\{\chi_3, \chi_4(\mathrm{ mod \ }q)}\\{\chi_1\chi_3 \textrm{ is nonprincipal}}}}\frac1{\textrm{ord} (\chi_1)\textrm{ord}(\chi_2) \textrm{ord} (\chi_3)\textrm{ord}(\chi_4)}\left|\sum_{a\leq N}\chi_1\chi_3(a)\right|\\
    &\ll N^{-1}(\log\log x)^4\sum_{p\leq x}\sum_{\substack{{q\leq x}\\{q\neq p}}}\sum_{\substack{{w|p-1}\\{t|\frac{p-1}w}}}\sum_{\substack{{u|q-1}\\{s|\frac{q-1}u}}}\sum_{\substack{{\chi_1(\mathrm{ mod \ }p)}\\{\chi_3(\mathrm{ mod \ }q)}\\{\chi_1\chi_3\textrm{ is nonprincipal}}}} \sum_{\substack{{\chi_2(\mathrm{ mod \ }p)}\\{\chi_4(\mathrm{ mod \ }q)}}} \frac1{\textrm{ord}(\chi_1)\textrm{ord}(\chi_3)}\cdot \frac1{\textrm{ord}(\chi_2)\textrm{ord}(\chi_4)}\left|\sum_{a\leq N}\chi_1\chi_3(a)\right|\\
    &\ll N^{-1}(\log\log x)^4\sum_{p\leq x}\sum_{\substack{{q\leq x}\\{q\neq p}}} \tau_3(p-1)\tau_3(q-1) \tau_2(p-1)\tau_2(q-1)\sum_{\substack{{\chi_1(\mathrm{ mod \ }p)}\\{\chi_3(\mathrm{ mod \  }q)}\\{\chi_1\chi_3\textrm{ is nonprincipal}}}}\frac1{\textrm{ord}(\chi_1)\textrm{ord}(\chi_3)}\left|\sum_{a\leq N}\chi_1\chi_3(a)\right|\\
    &=E_1 + E_2,
    \end{align*}
    where $E_1$ is the sum over $p$, $q$'s with $\tau_3(p-1)\tau_3(q-1) \tau_2(p-1)\tau_2(q-1)<\exp(c_8\sqrt{\log x})$, and $E_2$ is the sum over remaining $p$, $q$'s. Here, we let $c_8$ be a positive number smaller than $c_3$ in Lemma 2.5.
    By Lemma 2.5,
    $$E_1\ll x^2\exp(-c_9 \sqrt{\log x}).$$
    For $E_2$, we have
    \begin{align*}
    E_2 &\ll \sum_{p\leq x}\sum_{q\leq x} \tau_3(p-1)\tau_3(q-1) \tau_2(p-1)^2\tau_2(q-1)^2 \frac{\tau_3(p-1)\tau_3(q-1) \tau_2(p-1)\tau_2(q-1)}{\exp(c_8\sqrt{\log x})}\\
    &\ll \sum_{p\leq x}\tau_3(p-1)^2 \tau_2(p-1)^3 \sum_{q\leq x}\tau_3(q-1)^2\tau_2(q-1)^3 \exp(-c_8\sqrt{\log x})\\
    &\ll x^2\exp(-c_9 \sqrt{\log x}).
    \end{align*}
    The case when $\chi_2\chi_4$ is non-principal, is treated similarly and it also contributes $\ll x^2\exp(-c_9 \sqrt{\log x})$.

    Now, we find that the main contribution is when every character $\chi_1, \cdots, \chi_4$ is principal. In fact,
    \begin{align*}
    N^{-2}&\sum_{p\leq x}\sum_{\substack{{q\leq x}\\{q\neq p}}}\sum_{\substack{{w|p-1}\\{t|\frac{p-1}w}}}\sum_{\substack{{u|q-1}\\{s|\frac{q-1}u}}}\frac{\phi\left(\frac{p-1}w\right)}{p-1}\frac{\phi\left(\frac{p-1}{wt}\right)}{p-1}\frac{\phi\left(\frac{q-1}u\right)}{q-1}\frac{\phi\left(\frac{q-1}{us}\right)}{q-1}
    \left(N+O(1)+O\left(\frac Np\right)+O\left(\frac Nq\right)\right)^2\\
    &=N^{-2}\sum_{p\leq x}\sum_{\substack{{q\leq x}\\{q\neq p}}}\sum_{w|p-1}\sum_{u|q-1}\frac{\phi\left(\frac{p-1}w\right)}{w(p-1)}\frac{\phi\left(\frac{q-1}u\right)}{u(q-1)}\left(N^2+O(N)+O\left(\frac{N^2}p\right)+O\left(\frac{N^2}q\right)\right)\\
    &= \sum_{p\leq x}\sum_{\substack{{q\leq x}\\{q\neq p}}}\sum_{w|p-1}\sum_{u|q-1}\frac{\phi\left(\frac{p-1}w\right)}{w(p-1)}\frac{\phi\left(\frac{q-1}u\right)}{u(q-1)}+O\left(\frac{x^2}{N\log^2 x}\right)+O\left(\frac{x\log\log x}{\log x}\right)\\
    &=C^2\textrm{Li}^2(x) +O\left(\frac{x^2}{\log^E x}\right).
    \end{align*}
    Therefore,
    \begin{equation}
    N^{-2}\sum_{a\leq N}\sum_{b\leq N}\left(\sum_{\substack{{p\leq x}\\{p|a^n-b}\\{\textrm{for some $n$}}}}1-C\textrm{Li}(x)\right)^2= C^2\textrm{Li}^2(x) - C^2\textrm{Li}^2(x) + O\left(\frac{x^2}{\log^E x}\right).
    \end{equation}
    This completes the proof of Theorem 1.4.
    \end{proof}
    \subsection{Proof of Theorem 1.5}
    Following the definitions in ~\cite{LP},
    $$\Delta_q(n)=\# \{\textrm{cyclic factors } C_{q^v} \textrm{ in }(\Z/n\Z)^{*}: q^v || \ld(n) \},$$
    then
    \begin{equation}
    R(n)=\phi(n)\prod_{q|\phi(n)}\left(1-\frac1{q^{\Delta_q(n)}}\right).
    \end{equation}
    Let $\textrm{rad}(m)$ denote the largest square-free divisor of $m$. Let
    $$
    E(n)=\{a \in (\Z/n\Z)^{*} : a^{\frac{\ld(n)}{\textrm{rad}(n)}}\equiv 1 \textrm{ (mod $n$)}\},$$
    and we say that $\chi$ is elementary character if $\chi$ is trivial on $E(n)$. For each square free $h|\phi(n)$, let $\rho_n(h)$ be the number of elementary characters mod $n$ or order $h$. Then
    $$
    \rho_n(h) = \prod_{q|h}(q^{\Delta_q(n)} -1).$$
    For a character $\chi$ mod $n$, let
    $$
    c(\chi)=\frac1{\phi(n)} \sum_b^{ \ \ \ \ '} \chi(b), $$
    where the sum is over $\ld$-primitive roots in $[1,n]$. Then
    $$|c(\chi)|\leq \bar{c}(\chi),$$
    where
    $$\bar{c}(\chi)=\begin{cases} \frac1{\rho_n(\textrm{ord}(\chi))} &\mbox{ if $\chi$ is elementary,} \\ 0 &\mbox{ otherwise.} \end{cases}$$
    For the proof of above, see ~\cite[Proposition 2]{LP}.

    Let $X(n)$ be the set of non-principal elementary characters mod $n$. In ~\cite{L}, it is shown that
    \begin{equation}
    \sum_{a\leq y} N_a(x) = y\sum_{n\leq x}\frac{R(n)}n+B(x,y)+O(x\log x),
    \end{equation}
    where
    \begin{equation}
    B(x,y)=\sum_{n\leq x}\sum_{\chi\in X(n)}c(\chi)\sum_{a\leq y}\chi(a).
    \end{equation}
    Following the proof in ~\cite{LP},
    \begin{equation}
    |B(x,y)|\leq \sum_{d\leq x} |\mu(d)|S_d,
    \end{equation}
    where
    \begin{equation}
    S_d\ll \frac{xy}{d^2}\exp\left(3\sqrt{\frac{\log x}{\log\log x}}\right).
    \end{equation}
    We use this when $d$ is large.

    Let $\chi_{0,n}$ be the principal character modulo $n$. For positive integer $k$ and reals $w, z$, define
    $$
    F(k,z)=\sum_{\textrm{rad}(m)|k}\frac1m \sum_{\chi (\mathrm{ mod } \ k)}^{\ \ \ \ *} \bar{c}(\chi\chi_{0, km})\left|\sum_{a\leq z} \chi(a)\right|,$$
    $$
    T(w,z)=\sum_{k\leq w} F(k,z),$$
    $$
    S(w,z)=w\sum_{k\leq w} \frac1k F(k,z)=T(w,z)+w\int_1^w \frac1{u^2}T(u,z)du,$$
    and
    $$
    S_d\leq S\left(\frac xd, \frac yd\right).$$
    We want to estimate $S_d$ using the estimate of $T(w,z)$. Because of the integral in $S(w,z)$, we need an estimate of $T(u,z)$   for $1\leq u \leq w$. First, assume that $w\leq z^{\frac32}$. By applying P\'{o}lya-Vinogradov inequality, Li and Pomerance showed in ~\cite[Lemma 5]{LP} that \begin{equation}
    T(w,z)\ll w^{\frac32}\exp\left(3\sqrt{\frac{\log w}{\log\log w}}\right)\ll wz^{\frac34}\exp\left(3\sqrt{\frac{\log w}{\log\log w}}\right).
    \end{equation}
    Suppose now that $w> z^{\frac32}$. By H\"{o}lder inequality with $r=\lceil \frac{2\log w}{\log z} \rceil$,
    \begin{equation}
    T(w,z)^{2r}\leq A^{2r-1} B,
    \end{equation}
    where
    \begin{equation}
    A=\sum_{\substack{{k\leq w}\\{\textrm{rad}(m)|k}}}\frac1m\sum_{\chi (\mathrm{ mod } \ k)}^{\ \ \ \ *}\bar{c}(\chi\chi_{0, km})^{\frac{2r}{2r-1}},\end{equation}
    and
    \begin{equation}
    B=\sum_{\substack{{k\leq w}\\{\textrm{rad}(m)|k}}}\frac1m\sum_{\chi (\mathrm{ mod } \ k)}^{\ \ \ \ *}\left|\sum_{a\leq z}\chi(a)\right|^{2r}=\sum_{k\leq w}\frac{k}{\phi(k)}\sum_{\chi (\mathrm{ mod } \ k)}^{\ \ \ \ *}\left|\sum_{a\leq z}\chi(a)\right|^{2r}.
    \end{equation}
    Then by $0\leq \bar{c}(\chi\chi_{0, km})\leq 1$,
    \begin{equation}
    A\ll w\exp\left( 3\sqrt{\frac{\log w}{\log\log w}} \right).
    \end{equation}
    By large sieve inequality,
    \begin{equation}
    B \ll (w^2 + z^r)\sum_{a\leq z^r}(\tau_r'(a))^2
    \end{equation}
    If $w>z^{\frac32}$, then $w^{\frac1{2r}}>z^{\frac3{16}}$. By the method in Lemma 2.4, we have
    \begin{lemma}
    If $z>\exp(4.18\sqrt{\log w})$, then
    \begin{equation}
    T(w,z)\ll wz^{\frac{13}{16}}\exp\left(\sqrt{\log w} \left(f_1(4.18)+\frac{4.18}{4}+\epsilon\right)\right).
    \end{equation}
    \end{lemma}
    \begin{lemma}
    If $\exp(3.419906\sqrt{\log w})<z\leq\exp(16\sqrt{\log w})$, then
    \begin{equation}
    T(w,z)\ll wz^{\frac34}\exp\left(\sqrt{\log w}\left(f_1(3.419906)+\frac{3.419906}4+\epsilon\right)\right).
    \end{equation}
    \end{lemma}
    Therefore, by $S(w,z)=T(w,z)+w\int_1^w \frac1{u^2}T(u,z)du$,
    \begin{lemma}
    If $z>\exp(4.18\sqrt{\log w})$, then
    \begin{equation}
    S(w,z)\ll wz^{\frac{13}{16}}\exp\left(\sqrt{\log w} \left(f_1(4.18)+\frac{4.18}{4}+2\epsilon\right)\right).
    \end{equation}
    \end{lemma}
    \begin{lemma}
    If $\exp(3.419906\sqrt{\log w})<z\leq\exp(16\sqrt{\log w})$, then
    \begin{equation}
    S(w,z)\ll wz^{\frac34}\exp\left(\sqrt{\log w}\left(f_1(3.419906)+\frac{3.419906}4+2\epsilon\right)\right)+w\log w \cdot z^\frac78.
    \end{equation}
    \end{lemma}
    Suppose that $y>\exp(4.2\sqrt{\log x})$. If $\frac yd>\exp\left(4.18\sqrt{\log \frac xd}\right)$, then by Lemma 2.8,
    \begin{align*}
    S\left(\frac xd, \frac yd\right)&\ll \frac xd \left(\frac yd\right)^{\frac{13}{16}}\exp\left(\sqrt{\log \frac xd} \left(f_1(4.18)+\frac{4.18}{4}+2\epsilon\right)\right)
    \ll \frac{xy}{d^{\frac{29}{16}}y^{\frac{3}{16}}}  \exp\left(\sqrt{\log x} \left(f_1(4.18)+\frac{4.18}{4}+2\epsilon\right)\right)\\
    &\ll \frac{xy}{d^{\frac{29}{16}}} \exp\left(\sqrt{\log x} \left(-\frac{3}{16}\cdot 4.2 +f_1(4.18)+\frac{4.18}{4}+2\epsilon\right)\right)
    \end{align*}
    Since $-\frac{3}{16}\cdot 4.2 +f_1(4.18)+\frac{4.18}{4}<0$, the contribution of these $d$'s is
    $$
    \ll xy \exp\left(-c_{11}\sqrt{\log x}\right)$$

    If $\frac yd\leq \exp\left(4.18\sqrt{\log \frac xd}\right)$, then clearly $d\geq \exp\left(0.02\sqrt{\log x}\right)$. Then by (50), the contribution of these $d$'s is
    $$
    \ll xy\exp\left(-c_{12}\sqrt{\log x}\right).$$
    Therefore, we obtain the following weaker version of Theorem 1.5:
    \begin{theorem}
    If $y>\exp(4.2\sqrt{\log x})$, then
    \begin{equation}
    \frac1y \sum_{a\leq y} N_a(x) = \sum_{n\leq x} \frac{R(n)}n+ O\left(x\exp\left(-c_{13}\sqrt{\log x}\right)\right).
    \end{equation}
    \end{theorem}
    Now, assume that $\exp(3.42\sqrt{\log x})<y\leq \exp(16\sqrt{\log x})$, then for large $x$ and for $\frac yd \geq \exp\left(3.419907\sqrt{\log x}\right)$,
    $$
    \exp\left(16\sqrt{\log{\frac xd}}\right) \geq \frac{\exp\left(16\sqrt{\log x}\right)}d\geq \frac yd \geq \exp\left(3.419907\sqrt{\log x}\right)> \exp\left(3.419906\sqrt{\log \frac xd}\right).
    $$
    Thus by Lemma 2.9,
    \begin{equation}
    S\left(\frac xd, \frac yd\right)\ll \frac xd \left(\frac yd\right)^{\frac34}\exp\left(\sqrt{\log x}\left(f_1(3.419906)+\frac{3.419906}4+2\epsilon\right)\right)+\frac xd \log x \cdot \left(\frac yd\right)^{\frac78}.
    \end{equation}
    We sum this for $\frac yd \geq \exp\left(3.419907\sqrt{\log x}\right)$. By $-\frac{3.419907}4+f_1(3.419906)+\frac{3.419906}4<0$, the contribution is
    \begin{align*}
    &\ll xy \exp\left(\sqrt{\log x}\left(-\frac{3.419907}4+f_1(3.419906)+\frac{3.419906}4+2\epsilon   \right) \right) +\frac{xy\log x}{y^{\frac18}}\\
    &\ll xy \exp\left(-c_{14}\sqrt{\log x}\right).
    \end{align*}
    If $\frac yd < \exp\left(3.419907\sqrt{\log x}\right)$, then $d > \exp\left(0.000093\sqrt{\log x}\right)$. By (50), the contribution of these $d$'s is
    $$\ll xy \exp\left(-c_{15}\sqrt{\log x}\right).$$
    This completes the proof of Theorem 1.5.
    \section{Some Numerical Data and Calculus Remarks}
    The numerical values $3.419906$, $3.419907$, and $3.42$ are positive numbers greater than the unique solution to the equation $f_1(K)=0$. The values of $f_1(K)$ for these numbers are negative. Thus, $3.42$ in Theorem 1.1, 1.3, 1.5 can be replaced by any positive number greater than the unique solution to $f_1(K)=0$, which is numerically  $3.41990570065660\cdots$. Similarly, $4.2$ can be replaced by any positive number greater than the unique solution to $-\frac 3{16}K + f_1(K) + \frac K4 = 0$, which is numerically $4.17980309602625\cdots$. Also, $4.8365$ in Theorem 1.2, 1.4 can be replaced by any positive number greater than $4.83647702390563\cdots$.  The function $f_1(K)+\frac K4$ which can be simplified as $\frac1K\left(\log\left(\frac{K^2}2+1\right) +1\right)$, is a decreasing function for $K>0$ and it converges to $0$ as $K\rightarrow\infty$. The author used Wolfram Alpha for numerical calculations.
    \section{Further Developments}
    The author thinks that the corresponding normal order result for Theorem 1.5 could probably be done such as:
    If $y>\exp(4.8365\sqrt{\log x})$, then
    \begin{equation}
    \frac1y\sum_{a\leq y}\left(N_a(x) - \sum_{n\leq x}\frac{R(n)}n\right)^2 \ll x^2\exp\left(-c_{16}\sqrt{\log x}\right).
    \end{equation}
    The author also thinks that, in Theorem 1.1, 1.2, 1.3 and 1.4, if we replace $C\textrm{Li}(x)$ by $\sum_{p\leq x}\sum_{w|p-1}\frac{\phi\left(\frac{p-1}w\right)}{w(p-1)}$, then $(\log x)^D$, $(\log x)^E$ in the error terms may be replaced by $\exp\left(c_{17}\sqrt{\log x}\right)$.
           \flushleft

\end{document}